\title{Categoricity of the two sorted $j$-function}
\author{
        Adam Harris \\
        Oxford University
}
\date{\today}

\documentclass[12pt,a4paper]{article}
\usepackage{amssymb ,amsfonts, amsmath, amscd, amsthm, verbatim, natbib, graphics, tikz}
\usepackage{graphics}                 
\usepackage{color}                    
\usepackage{hyperref}                 
\usepackage{appendix}
\usepackage{tikz-cd}
\usepackage{fullpage}

\input xy
\xyoption{all}

\newtheorem{thm}{Theorem}[section]
\newtheorem{cor}[thm]{Corollary}
\newtheorem{lem}[thm]{Lemma}
\newtheorem{pro}[thm]{Proposition}

\theoremstyle{remark}
\newtheorem{rem}[thm]{Remark}

\theoremstyle{remark}

\theoremstyle{definition}
\newtheorem{dfn}[thm]{Definition}

\theoremstyle{definition}

\def\fsl{\mathfrak{sl}}
\def\tc{\widetilde{\mathbb{C}}}
\def\M{\mathbb{ M}}
\def\G{\mathbb{ G}}

\def\N{\mathbb{ N}}

\def\Z{\mathbb{ Z}}
\def\Q{\mathbb{ Q}}

\def\C{\mathbb{ C}}

\def\P{\mathbb{ P}}
\def\A{\mathbb{ A}}
\def\H{\mathbb{ H}}

\def\*{{}^{*}}

\def\cL{\mathcal{L}}
\def\cM{\mathcal{M}}
\def\cC{\mathcal{C}}
\def\cH{\mathcal{H}}

\def\cF{\mathcal{F}}

\def\cK{\mathcal{K}}
\def\cN{\mathcal{N}}

\def\Lra{\Leftrightarrow}
\def\ra{\rightarrow}

\def\lora{\longrightarrow}
\def\gl2q{\text{\textnormal{GL}}_2^+(\Q)}
\def\ggl2{\text{\textnormal{GL}}_2}
\def\GL{\text{\textnormal{GL}}}

\def\SL{\text{\textnormal{SL}}}
\def\sl2{\text{\textnormal{SL}}_2}
\def\psl2{\text{\textnormal{PSL}}_2}
\def\sv2{\text{SZ}_2}
\def\z2{\text{Z}_2}

\def\pgl2{\text{\textnormal{PGL}}_2}

\def\m2z{\text{M}_2(\Z)}

\def\acl{\text{\textnormal{acl}}}
\def\cl{\text{\textnormal{cl}}}
\def\tp{\text{\textnormal{tp}}}

\def\qftp{\text{\textnormal{qftp}}}

\def\tor{\text{\textnormal{Tor}}}
\def\Loc{\text{\textnormal{Loc}}}
\def\Th{\text{\textnormal{Th}}}
\def\T{\text{\textnormal{T}}}
\def\End{\text{\textnormal{End}}}
\def\trdeg{\text{\textnormal{trdeg}}}
\def\MT{\text{\textnormal{MT}}}
\def\Hg{\text{\textnormal{Hg}}}

\def\CycSub{\text{\textnormal{CycSub}}}

\def\Cov{\text{\textnormal{Cov}}}

\def\Gal{\text{\textnormal{Gal}}}

\def\Aut{\text{\textnormal{Aut}}}

\def\Fet{\text{\textnormal{F\'et}}}

\def\dim{\text{\textnormal{dim}}}
\def\dcl{\text{\textnormal{dcl}}}

\def\Fet{\text{\textnormal{F\'et}}}

\def\g0n{\Gamma_0(N)}

\def\lora{\longrightarrow}

\def\ggl2{\textrm{GL}_2}
\def\GL{\textrm{GL}}
\def\sl2{\textrm{SL}_2}

\def\sv2{\textrm{SZ}_2}
\def\z2{\textrm{Z}_2}
\def\pgl2{\textrm{PGL}_2}

\def\m2z{\textrm{M}_2(\Z)}

\include{head}

\begin{document}

\maketitle

\begin{abstract}
We show that a natural, two sorted $\cL_{\omega_1,\omega}$ theory involving the modular $j$-function is categorical in all uncountable cardinaities. It is also shown that a slight weakening of the adelic Mumford-Tate conjecture for products of elliptic curves is necessary and (along with a couple of other results from arithmetic geometry) sufficient for categoricity.
\end{abstract}

\section{Introduction}

It is a general theme in showing natural mathematical structures have nice model theoretic properties, that the required model theoretic properties (for example homogenity and atomicity properties required for categoricity or stability), correspond directly to known theorems of classical mathematics. The fact that these theorems correspond to model theoretic notions, might be seen as providing some new justification as to why they should hold. \newline

Consider the upper half plane $\H$ with the group $\gl2q$ acting on it via
$$\begin{pmatrix} a&b\\ c&d \end{pmatrix} \tau = \frac{a \tau + b}{c \tau + d},$$
along with the modular $j$-function going from $\H$ onto the complex numbers $\C$. We wish to show that the $j$-function is a `canonical mathematical object' by showing that it has a natural, categorical axiomatization in some language. Some possible intuition behind this setup is that $\H$ is nearly the universal cover of $\C$, which is complex analytically isomorphic to $\sl2(\Z) \backslash \H$ via $j$. $\gl2q$ is a group of analytic automorphisms on the cover $\H$, and we might want to study the action of each $\alpha \in \gl2q$ down in $\C$. Here, the Shimura variety $\C$ is the moduli space of elliptic curves, and the action of $\gl2q$ on its cover gives additional information regarding isogenies. \newline

Consider a first order language $\cL$ for (two sorted) structures of the form
$$\cM = \langle \langle H ; \{g_i \}_{i \in \N}  \rangle, \langle F,+,\cdot , 0, 1 \rangle, j:H \ra F \rangle $$ where the structure $ \langle H ; \{g_i \}_{i \in \N}\rangle$ is a set $H$ with countably many unary function symbols, $ \langle F;+,\cdot, 0, 1 \rangle$ is an algebraically closed field of characteristic zero, and the function $j$ goes from $H$ to $F$. \newline

Let $\Th(j)$ be complete first order theory of  standard $j$-function in the language described above i.e. the first order theory of the `standard model' $$\C_j:=\langle \langle \H, G \rangle , \langle \C , + , \cdot , 0 , 1 \rangle , j: \H \ra \C \rangle.$$

We also define the $\cL_{\omega_1, \omega}$ axiom
$$\textrm{StandardFibres}: \ \ \ \forall x \forall y \ \left(  j(x)=j(y)\ra\bigvee_{\gamma \in \sl2(\Z)}x=\gamma(y) \right)$$
which fixes a fibre of $j$ to be an $\sl2(\Z)$-orbit.
Let $\T_{\omega_1,\omega}(j):=\Th(j) \cup \textrm{StandardFibres}$. We sometimes abbreviate $\textrm{StandardFibres}$ by $\textrm{SF}$. \newline

In this note, we prove the following:
\begin{thm}\label{thmcat}
The theory of the $j$-invariant $\T_{\omega_1, \omega}(j)+\trdeg(F) \geq \aleph_0$ has a unique model (up to isomorphism) in each infinite cardinality. 
\end{thm}
Technically, the above means that given two models $$\cM = \langle \cH  , \cF, j:H \ra F \rangle \  , \  \cM' = \langle \cH'  , \cF', j':H' \ra F' \rangle$$ of the same infinite cardinality, where 
$$\cH=\langle H ; \{g_i \}_{i \in \N}  \rangle  \textrm{ and } \cF=\langle F,+,\cdot , 0, 1 \rangle,$$ there are isomorphisms $\varphi_H$ and $\varphi_F$ such that the following diagram commutes:
$$
\xymatrix{
\cH \ar[r]^{\varphi_H} \ar[d]^j &    \cH' \ar[d]^{j'}
\\
\cF \ar[r]^{\varphi_F}  & \cF'
}
$$
Since the theory of algebraically closed fields is categorical in all uncountable cardinalities, we may identify the two fields at the bottom with an arbitrary isomorphism and just think about two `coverings' of an algebraically closed field. \newline

The main result from outside of model theory required in proving the above is an instance of the Mumford-Tate conjecture for products of elliptic curves. We actually go further and show that (a slight weakening of) this result is necessary for $\aleph_1$-catgeoricity.

\subsection{Strategy}

For categoricity we will ultimately appeal to the theory of quasiminimal excellence as in \cite{kirby2010quasiminimal} and \cite{bays2012quasiminimal} . The basic idea is as follows:\newline

We wish to build an isomorphism between two models $\cM$ and $\cM'$. Since the first order theory is complete we may identify $\dcl^{\cM}(\emptyset)$ with $\dcl^{\cM'}(\emptyset)$ and by quantifier elimination this means we always work `over the special points'. We build an isomorphism by a back an forth argument which may be seen as an abstraction of the method used to extend automorphisms in algebraically closed fields in classical algebra. At a certain stage in the back and forth argument, we assume that we have a partial isomorphism
$$\langle \bar{x} \rangle \cong \langle \bar{x}' \rangle $$
where $\langle \bar{x} \rangle$ is the substructure generated by $\bar{x}$. Then we take a new element $y \in \cM$ and we wish to find an element $y' \in \cM'$ which satisfies the same quantifier free formulae (with parameters in $\bar{x}$) as $y$. Note that we have used the fact that
$$\qftp(\bar{a}) = \qftp(\bar{b}) \Lra \langle \bar{a} \rangle \cong \langle \bar{b} \rangle.$$
From here the argument is two pronged:
We first show that we can realize the field type of a finite subset of a Hecke orbit over any parameter set (this is a result regarding algebraicity of modular curves, which gives a strong connection between the two sorts \ref{lem_locus_irreducible}), and then we show that all the information in the type is contained in this `finite' part (an adelic Mumford-Tate (open image) kind of result, \ref{amtapp}). In particular, if we take a point $\tau \in \H$, then $\tau$ corresponds to some elliptic curve $E$, and the type of $\tau$ is determined by algebraic relations between the torsion points of $E$, i.e. it is determined by the Galois representation on the Tate module of $E$. An adelic version of Serre's open image theorem (i.e. the adelic Mumford-tate for elliptic curves) tells us that there are no unexpected relations between torsion points as $N$ gets large, and so the type is determined by relations between a finite subset of the $N$-torsion and the type is (nearly) isolated.\newline

\subsection{Notation}
\begin{itemize}
\item For an abelian variety $A$ define $A[N]$ to be the $N$-torsion, $\tor(A):=\cup_N A[N]$, $T_l(A)$ is the $l$-adic Tate module of $A$ and $T(A):=\prod_l T_l$ is the Adelic Tate module. We also let $\MT(A)$ be the Mumford Tate group of $A$ and $\Hg (A)$ be the Hodge group (sometimes called the special Mumford-Tate group - see \cite[5.8]{moonen2004introduction});
\item For $\tau \in \H$ we define $E_{j(\tau)}$ to be (the $\C$ points of) an elliptic curve which has $j$-invariant $j(\tau)$ and is defined over $\Q(j(\tau))$. For example, we may take the elliptic curve defined by 
$$y^2=4x^3-c x - c \ \ \ \ \ c=\frac{27j(\tau)}{j(\tau)-1}$$
(all fields are characteristic 0 unless stated othersise). For $\bar{\tau}= \langle \tau_1,...,\tau_n \rangle \in \H$ we associate the corresponding abelian variety $$A_{j(\bar{\tau})}:=E_{j(\tau_1)}\times \cdots \times E_{j(\tau_n)}$$ defined over $\Q(j(\tau_1),...,j(\tau_n))$;
\item For a subset $G'$ of $G$ we will write $j(g \tau)$ for $\{j(g \tau) \ | \ g \in G' \}$. We note that the scalar matrices act trivially on $\H$, and define $G_N$ to be those matrices in $G$ such that if you multiply through by an integer to clear denominators in all entries, then the determinant is $N$. Equivalently $G_N=\Gamma  \begin{pmatrix} N&0 \\ 0&1 \end{pmatrix} \Gamma$ where $\Gamma:=\sl2(\Z)$. We define $G_0:= \emptyset$ and we also sometimes denote $G$ by $G$. For $\tau \in \H$, $j(G \tau)$ is called a \textit{`Hecke orbit'};
\item If $s \in H$ is (the unique element) fixed by some $g \in G$, then $s$ is called {\it`special'}. The set of all special points in $H$ is denoted $S$. In this situation, we also say that $j(s)$ is special;
\item For a field $K$ we let $G_K$ be the absolute Galois group $\Gal(\bar{K} / K)$, and $K^{cyc}:=K(\tor(\C^{\times}))$ the field obtained by adjoining all roots of unity to $K$;
\item For $(x_1,...,x_n) \in \C^n$ and $K$ a subfield of $\C$, we let $\Loc( (x_1,...,x_n) / K)$ be the minimal (irreducible) algebraic variety over $K$ containing $(x_1,...,x_n)$ (i.e. the `Weil locus' of $(x_1,...,x_n)$ over $K$);
\item We note that the centre of $\gl2q$ (scalar matrices) act trivially on $\H$, so we will actually consider the group 
$$G:=\gl2q^{ad}=\gl2q / Z(\gl2q).$$
This is not essential, but this gives us a faithful action and makes our axiomatizatisation cleaner.
\end{itemize}

\section{Geometry and arithmetic}

In this section we include all the background results in arithmetic geometry which are needed to set the scene, and also those which are required to prove our model theoretic result. For the relevant background on modular curves we refer the reader to \cite{diamond2005first}.

\subsection{The `Galois closure' of $Y_0(N)$}\label{secgalcl}

Consider the congruence subgroup
$$G(N):=\left\{  \begin{pmatrix} a&b \\ c&d \end{pmatrix} \in \SL_2(\Z) \ | \ a \equiv b \mod N, \ c \equiv d \equiv 0 \mod N \right\},$$
which is a normal subgroup of $\SL_2(\Z)$, and let
$$Y_0'(N):=G(N)\backslash \H.$$

\begin{pro}
A disjoint union of two copies of $Y_0'(N)$ is the moduli space of isomorphism classes of elliptic curves with a pair of disjoint cyclic subgroups of order $N$. Considered as a covering space of $Y_0'(1) \cong \C$ we have
$$\Aut_{\Cov}(Y'_0(N) / Y_0'(1)) \cong \psl2(\Z / N \Z).$$
\end{pro}
\begin{proof}
To see this, we note that if two disjoint cyclic subgroups of $E[N]$ are generated by $B_1$ and $B_2$, then $B_1$ and $B_2$ are a basis for $E[N]$. We then forget which cyclic subgroups they came from. The reason you need two copies is due to the Weil pairing inducing an arithmetic action on a quadratic subfield $\Q(\eta_N)$ of $\Q(\mu_N)$.
\end{proof}

Let $\Gamma: = \SL_2(\Z)$. In general, a tuple $\bar{g} =(g_1,...,g_n) \subset G$ determines a congruence subgroup of $\Gamma$
$$\Gamma_{\bar{g}}  := g_1^{-1} \Gamma g_1 \cap \cdots g_n^{-1} \Gamma g_n.$$
Notice that this subgroup is independent of the ordering of the tuple and therefore it is really the set $\{g_1,...,g_n \}$ which determines the group $\Gamma_{\bar{g}} $. The quotient
$$ \Gamma_{\bar{g}}   \backslash \H$$
is an algebraic curve $V_{\bar{g}} $ (see for example \cite[Chapter 2]{diamond2005first}). If for the tuple $\bar{g}$, we take $g_1 = 1$ and the rest of the $g_i$ a set of ($\psi(N)$ many) coset representatives for $\Gamma \backslash G_N$ then it is easy to see that $$\Gamma_{\bar{g}} = G(N).$$

Via the modular polynomial $\Phi_N(X,Y) \in \Z[X,Y]$, the curve $Y_0(N)$ has a model $Y_0(N)_{\Q}$ over $\Q$. Note that the function field of $Y_0(N)$ is
$$\C(Y_0(N))=\C(j,j\circ N)$$
and the function field of $Y_0(N)_{\Q}$ is
$$\Q(Y_0(N))=\Q(j,j\circ N).$$
Now consider the curve with function field
$$ \C(j, \{j\circ g\}_{g \in G_N} )=\textrm{Normalclosure}(\C(j,j \circ N))$$
and, via $\Phi_N$ again, consider its $\Q$-model $\overline{Y_0(N)}_{\Q}$. Now consider the map
$$\tau \mapsto (j(\tau) , j(g_1 \tau),...,j(g_{\psi(N)} \tau))$$
for $g_i \in \Gamma \backslash G(N)$. Then we have
$$\Phi_N(j(\tau), j(g _i \tau )) = 0$$
and the image is an algebraic curve, isomorphic to $Y_0'(N)$, which lies inside $\overline{Y_0(N)}$.

At this point, we note the following:
\begin{lem}\label{lem_locus_irreducible}
Given $g_1,...,g_n \in G$, the image of $\H$ under the map
$$f(z)=(j(g_1z),...,j(g_nz))$$
is an geometrically irreducible algebraic curve defined over $\Q(j(S))$ (where $S$ is the set of special points).
\end{lem}

\begin{proof}
The image $f(\H)$ is a strongly special subvariety of $\C^n$ and is therefore geometrically irreducible and contains infinitely many special points.
\end{proof}

By the above discussion, we see that the image of the map
$$\tau \mapsto (j(\tau) , j(g_1 \tau),...,j(g_{\psi(N)} \tau))$$
for $g_i \in \Gamma \backslash G(N)$, is a $\Q(\eta_N)$-model of $Y'_0(N)$, which we denote by $Y'_0(N)_{Q(\eta_N)}$. $Y'_0(N)_{Q(\eta_N)}$ and its $\Gal(\Q(\eta_N) / \Q)$-conjugate are the irreducible components of $\overline{Y_0(N)}_{\Q}$, however in this note we will only be interested in the component $Y'_0(N)_{Q(\eta_N)}$ coming from the map $\tau \mapsto (j(\tau) , j(g_1 \tau),...,j(g_{\psi(N)} \tau))$, since it is these maps, and these curves which control the type of $\tau$ with respect to the model-theory later on. For this reason, we will be interested in the pro-\'etale cover
$$\widetilde{\C}:= \varprojlim_N Y_0'(N)$$
which exists in the category of schemes, and the projections are just the natural rational maps (defined over $\Q(\eta_N)$) induced by inclusions of function fields. \newline

Also taking into account the above discussion, we have the following
\begin{pro}
 $$\Gal(\Q(j , \{j\circ g\}_{g \in G_N} ) / \Q(j)) \cong \pgl2(\Z / N \Z)$$
and
$$\Gal(\Q^{cyc}(j ,  \{j\circ g\}_{g \in G_N} ) / \Q^{cycl}(j)) \cong \psl2(\Z / N \Z).$$
\end{pro}

If $N$ is not prime then there are nontrivial square roots of unity mod $N$. This means that $\pgl2(\Z / N \Z)$ is strictly larger than $\psl2(\Z / N\Z)$ and in this case there is an element $\pgl2(\Z / N \Z)$ which switches between the two connected components of $\overline{Y_0'(N)}_{\Q}$, but as mentioned above, only one connected component is ever seen with respect to the model theory. \newline

Define $$\pi_1':=\varprojlim_N \Aut_{\Fet}(Y_0'(N) / Y_0'(1)) \cong \psl2(\hat{\Z}).$$
It seems appropriate to use this notation because $\pi_1'$ is nearly the \'etale fundamental group of $\P_1(\C) - \{0,1, \infty \}$, and we want to think of it as one (we have taken the limit over a directed system which is a proper subset of the finite \'etale covers of $X_0'(1)$). A fibre in $\widetilde{\C}$ above a point in $\C$ is a $\pi_1'$-torsor.  \newline

For an elliptic curve $E/K$, let $\CycSub(E,N)$ be the set of cyclic subgroups of $E(\bar{K})$ of order $N$. We have the following:
\begin{pro}\label{definable cyclic}
If $K$ is a subfield of $\C$ and $E / K$ is an elliptic curve such that $\End(E)\cong \Z$, then there is a bijection
$$\eta: \{ j(E/C) \ | \ C \in \CycSub(E,N) \}\longleftrightarrow \{C \ | \ C \in \CycSub(E,N) \}$$
which is definable in the language of rings with parameters from $K$.
\end{pro}
\begin{proof}
Follows from \ref{pro_uniformrep_subgroup_isog} and \ref{pro:finitesubgroupsisomorphiccurves}.
\end{proof}

 By the above and \ref{lem_locus_irreducible}, we end up with the following key model theoretic statement:
\begin{lem}\label{reciprocity}
Let $\cM$ and $\cM'$ be models of $\Th(j)$ as above and let $\tau \in H$, $\tau' \in H'$ such that $j(\tau)=j'(\tau')$. Then there is $\sigma \in \pi_1'$ such that $\sigma \eta (j(g \tau))=j'(g \tau')$ for all $g \in G$ i.e. in the above situation we may associate $\tau$ and $\tau'$ to $\tilde{\tau}$ and $\tilde{\tau}'$ in the pro-\'etale cover $\tc$.
\end{lem}

\subsection{The action of Galois on $\tc$}

Let $\tilde{p} :\tc \ra \C$ be the covering map. Given a point $z \in \C$, there is also an action of Galois on $\tilde{p}^{-1} (z)$. We can translate this into an action on the $j$ values corresponding to cyclic subgroups via the following.
\begin{pro}
Let $E$ be a non-CM elliptic curve over $\Q(j(E))$ and let \linebreak $\sigma \in \Gal(\C / \Q(j(E))$. Then for all $C_1,C_2 \in \CycSub(E,N)$
$$\sigma j(E/C_1)=j(E / C_2) \textrm{ iff } \sigma C_1 = C_2.$$
\end{pro}
\begin{proof}
This follows from \cite[2, \S3 Proposition 3]{lang1987elliptic}, and \ref{definable cyclic}.
\end{proof}

\subsection{The induced `representation'}
 The fibre $\tilde{p}^{-1}(z) \subset \tc$ above the point $z \in \C$ is a $\pi_1'$-torsor. The fibre is also acted on by $\Gal(\C / \Q(z))$, and the actions commute giving us a homomorphism $$\rho: \Gal(\C / \Q(z)) \lora \pi_1'.$$

\subsection{Galois representations on Tate modules}

Let $A$ be an abelian variety of dimension $g$ defined over a field $K$. Then there is a continuous Galois representation
$$\rho: G_K \longrightarrow \Aut(T(A)) \cong \GL_{2g}(\widehat{\Z})$$
on the Adelic Tate module of $A$. \newline

The Adelic Mumford-Tate conjecture states that for any abelian variety $A$ over a number field $K$, the image $\rho(\Gal(\bar{K} / K))$ is open in the $\widehat{\Z}$ points of the Mumford-Tate group $\MT(A)(\widehat{\Z})$. For example, this is a theorem of Serre (\cite{serre1986resume}) if $\End(A ) = \Z$ and $\dim (A)$ is odd or equal to 2 or 6. For the definition of the Mumford-Tate group $\MT(A)$ please see \cite{moonen2004introduction}. The Hodge group of $A$, is defined as
$\Hg(A): = \MT(A) \cap \SL_V$ where $V=H^1(A(\C),\Q)$. $\MT(A)$ is the almost direct product of $\Hg(A)$ and $\G_m$. If the reader doesn't know any Hodge theory then they can just think of the Hodge group as a notational device. In the case of a non-CM elliptic curve $E$, $\Hg(E)(\widehat{\Z})\cong \sl2(\widehat{\Z})$. The Hodge group behaves well with respect to products of elliptic curves (i.e. the Hodge group of the product is the product of the Hodge groups \cite{imai1976hodge}). \newline


As in \cite[3.4]{ribet1975}, we say that a profinite group satisifes the `commutator subgroup condition' if for every open subgroup $U$, the closure of the commutator subgroup $[U:U]$ of $U$, is open in $U$.

\begin{thm}\label{amt}
Let $A$ be an abelian variety defined over $K$, a finitely generated extension of $\Q$ or a finitely generated extension of an algebraically closed field $k$, such that $A$ is a product of $r$ non-isogenous elliptic curves (with $j$-invariants which are transcendental over $k$ in the second case). Then the image of the Galois representation on the Tate module of $A$ is open in $\Hg(A)(\hat{\Z})$.
\end{thm}
\begin{proof} For $K$ a number field and $r \leq 2$, this was done by Serre (\cite[\S 6]{serre1971prop}), and Ribet reduced the problem to the case $r=2$ by noting that $\SL_2(\hat{\Z}$) satisfies the commutator subgroup condition, see \cite[3.4]{ribet1975}. So we just need to check the function field case (i.e. the second case in the hypothesis of the theorem). If we have a product $A=E_1 \times \cdots \times E_r$ of $r$ non-isogenous elliptic curves with transcendendal $j$-invariants and $K = \Q(j(E_1),...,j(E_r))$, then the image of the Galois representation
$$\rho : G_K \ra \Hg(A)( \hat{\Z})$$
is open. This follows from the number field case and the `spreading out' argument of \cite[Remark 6.12]{pink2005combination}. Now we lift this result up to over an algebraically closed field by the theory of regular field extensions:\newline

Consider two non isogenous elliptic curves $E_1$, $E_2$ with transcendental $j$-invariants and let $K=\Q(j(E_1),j(E_2))$. Then by the  above the image of
$$\rho: G_{K^{cyc}} \lora \Hg(E_1 \times E_2)(\Z) \cong \SL_2(\hat{\Z} ) \times \SL_2(\hat{\Z})$$
is open, and there is a finite extension $K \subseteq L$ such that
$$K^{cyc}(\tor(E_1)) \cap K^{cyc}(\tor(E_2)) = L^{cyc}.$$
 Since $L^{cyc}(\tor(E_1))$ and $L^{cyc}(\tor(E_1))$ are both Galois extensions of $L^{cyc}$, which intersect in $L^{cyc}$, they are linearly disjoint over $L^{cyc}$ and are therefore free over $L^{cyc}$. Also, the extensions
$$ L^{cyc}(\tor(E_i)) / L^{cyc} $$
are regular (see for example \cite[6,\S3]{lang1987elliptic}), so if $F \subset \C$ is a countable algebraically closed field then
$$L^{cyc}(\tor(E_i)) \cap F = L^{cyc} \cap F.$$ Now since the $L^{cyc}(\tor(E_i))$ are both regular extensions of $L^{cyc}$, and they are free from each other over $L^{cyc}$, the compositum $L^{cyc}(\tor(E_1), \tor(E_2))$ is a regular extension of $L^{cyc}$ ( \cite[VIII,\S4, Corollary 4.14]{lang2002algebra}).Then
$$\Gal(FL(\tor(E_1) , \tor(E_2)) / FL) \cong  \Gal(L^{cyc}( \tor(E_1) , \tor(E_2)) / L^{cyc} ( \tor(E_1) , \tor(E_2)) \cap FL)$$
but since the extension is regular we have
$$L^{cyc} ( \tor(E_1) , \tor(E_2)) \cap FL^{cyc} = L^{cyc} \cap FL $$
and the result follows since we know that
$\Gal(L^{cyc}( \tor(E_1) , \tor(E_2)) / L^{cyc} )$
is an open subgroup of $\SL_2(\hat{\Z}) \times \SL_2(\hat{\Z})$ and $L$ is a finite extension of $K$.
\end{proof}

\begin{cor}\label{amtcor}
Let $E_i,...,E_r$ be pairwise non-isogenous elliptic curves  defined over $K$, a finitely generated extension of $\Q$ or a finitely genrated extension of an algebraically closed field $k$ where in this case the $j$-invariants  are transcendental over $k$. Then
$$[K(\tor(E_i))\cap K(\cup_{j \neq i} \tor(E_j)): K^{cycl}] \textrm{ is finite.}$$
\end{cor}

\begin{lem} \label{lem abelian}
The extension
$$\Q^{cycl}(j(S)) / \Q^{cycl}$$
is abelian.
\end{lem}
\begin{proof}
Let $\tau \in \H$ be such that $E_{j(\tau)}$ has complex multiplication by an imaginary quadratic field $K$. Then $\Q(j(G \tau)$ contains $K$, and is a subset of the field obtained by adjoining $j(\tau)$ and the $x$-coordinates of the torsion of $E_{j(\tau)}$ to $\Q$ (which is an abelian extension of $K$ by the theory of complex multiplication - see for example \cite[p135]{silverman1994advanced}). The result now follows since $\Q^{cycl}$ contains all imaginary quadratic fields and the compositum of abelian extensions is abelian.
\end{proof}

\begin{cor}
Given an abelian variety $A \cong E_1 \times \cdots \times E_n$ defined over $K$ a finitiely generated extension of $\Q$ where $E_i$ are non-isogenous elliptic curves with $\End(E_i)\cong \Z$ we have
$$[\Q^{cycl}(j(S))\cap K(\tor(A)) : K^{cycl} ] \textrm{ is finite.}$$
\end{cor}
\begin{proof}
$(\SL_2 \times \SL_2)(\hat{\Z})$ satisfies the commutator subgroup condition, since the Lie algebra $\fsl \times \fsl$ is its own derived algebra. The result follows since the intersection $K^{cycl}(j(S))\cap K(\tor(A))$ has to be an abelian extension of $K^{cyc}$ by \ref{lem abelian}.
\end{proof}

The following is now immediate:
\begin{cor}\label{amtapp}
Given an abelian variety $A$ as above, and an elliptic curve $E / K$ with $\End(E) \cong \Z$, such that $E$ is non-isogenous to all of the $E_i$ in the product $A$. Let $L:=K^{cycl}(\tor(A))$, then the image of the representation
$$\rho : G_L \lora \Aut(T(E))$$
is open in $\SL(\hat{\Z})$.
\end{cor}



\section{Description of the types}

If $\langle H,F \rangle \models \Th(j)$, then we let $\tp_{H}(\tau)$ stand for the type of $\tau$ in the group action sort only, and $\tp_F(z)$ be the type of $z$ in the field sort only.

\begin{pro}\label{fielddet}
For $\bar{\tau} \subset H - S$, $\qftp(\bar{\tau})$ is determined by the field type of its Hecke orbit.
\end{pro}
\begin{proof}
If $\bar{\tau} \cap S=\emptyset$ then all that can be said about $\bar{\tau}$ with quantifier free formulae in the $H$ sort is whether or not any of the coordinates of $\bar{\tau}$ are are related by elements of $G$. If so, then this is expressible down in the field sort by `modular polynomials' i.e.
$$\exists g \in G  \ g \tau = \tau' \ {\rm iff } \ \exists N \in \N \ \ \Phi_N(j(\tau) , j(\tau ' ))=0$$
(see for example  \cite[5,\S3]{lang1987elliptic}).

\end{proof}

The type of $\tau$ is determined by $$\bigcup_{(g_1,...,g_n) \subset G} \Loc((j(g_1 \tau),...,j(g_n \tau)) / \dcl(\emptyset) \cap F)$$ where here we are taking the union over all tuples $(g_1,...,g_n) \subset G$. Note that $\Q(j(S)) \subset \dcl(\emptyset)$, so we are always at least working over the special points. In fact, by quantifier elimation $\dcl(\emptyset) = \Q(j(S))$, but we do not know this at this point. As mentioned in \ref{secgalcl}, it is actually the set $\{g_1,...,g_n \}$ which determines an algebraic curve, and not the tuple, so we see that for a tuple $\bar{\tau} \subset H$, $\qftp(\bar{\tau})$ is determined by
$$\bigcup_{ \{ g_1,...,g_n \} \subset G} \Loc((j(g_1  \bar{\tau}),...,j(g_n \bar{\tau})) / \dcl(\emptyset) \cap F).$$


\section{Realising types}

As mentioned in the introduction, the following result will allow us to realise the field type of a finite subset of a Hecke orbit over any set. Note that the statement of \ref{lem_locus_irreducible} appears in $\Th(j)$.


 \begin{cor}\label{realise}
Given two models $\cM$ and $\cM'$ of $\T_{\omega_1 , \omega}(j)$,  non-special $\tau \in H$, $L \subseteq F$ and any finite subset $\langle g_1,...,g_n \rangle \subset G$ of $\qftp(\tau / L) $, we may find $\tau' \in H'$ realising $\qftp(j(g_1 \tau),...,j(g_n \tau) / L)$.
\end{cor}

\begin{proof}
For non-special $\tau$, by \ref{fielddet} $\qftp(j(g_1 \tau),...,j(g_n \tau))$ is determined by $$ \Loc(( j(g_1 \tau),...,j(g_n \tau) ) / L \cup \dcl(\emptyset) \cap F)$$ (in the theory $ACF_0$,  $\qftp(\bar{a} / K)$ is determined by $\Loc(\bar{a} / K)$). We know that \linebreak
  $\Loc(( j(g_1 \tau),...,j(g_n \tau)) / L \cup \dcl(\emptyset) \cap F)$ is a subvariety of the algebraic curve $C:=f(\H)$ given by the lemma. This curve $C$ is defined by over $\Q(j(S)) \subset \dcl(\emptyset)$ and therefore the statement of the lemma (and in particular that the function $f'(z):=(j'(g'_1z),...,j'(g'_nz))$ is onto the curve $C$) appears in $\Th(j)$. Since $f'$ maps onto $C$, it maps onto the subvariety $\Loc(( j(g_1 \tau),...,j(g_n \tau)) / L \cup \dcl(\emptyset) \cap F)$.
\end{proof}

\section{Axiomatisation}

For a structure $\cM$ we let $\Th(\cM)$ denote its first order theory.

If $\langle g_1,...,g_n \rangle \subset G$, then let $\Psi_{( g_1,...,g_n )}$ be a first order sentence expressing the statement of
$$\left( \forall x \  ( j(g_1 x),...,j(g_n x) ) \in V_{(g_1,...,g_n)}  \right)  \bigwedge \forall \bar{v} \in V_{(g_1,...,g_n)} \left( \ \exists x  ( j(g_1 x),...,j(g_n x) ) = \bar{v} \right)$$
 for some algebraic curve $V_{(g_1,...,g_n)}$ (provided by \ref{lem_locus_irreducible}).  \newline 

Let $R \subseteq G^3$ be the ternary relation corresponding to multiplication in $G$ i.e. 
$$R(g_1,g_2,g_3) \textrm{ iff } g_1.g_2 = g_3.$$

Now let $T_H$ be a first order axiom scheme which includes
\begin{itemize}
\item An axiom scheme stating that the action of $G$ on $H$ is faithful;
\item For all $(g_1,g_2,g_3) \in R$, the sentence
$$\forall x , g_1.g_2 x =g_3 x;$$
\item For each $g \in G$ fixing a point of $\H$, an axiom stating that $g$ has a unique fixpoint in $H$.
\end{itemize}
The first two axioms ensure that (up to isomorphism) we have an action of the correct group, since having a faithful action means that
$$\forall x ,(g_1.g_2)( x) = g_3 (x)  \textrm{ iff }   g_1.g_2 = g_3.$$
 It is easy to see that the theory $T_H$ is a strongly minimal theory, and therefore is complete has a unique model in each uncountable cardinality. The model-theoretic algebraic closure operator gives a trivial pregeometry and we therefore have a notion of independence on the sort $H$. A tuple $( \tau_1,..., \tau_r) \subset H$ is said to be $G$-\textit{independent} if it is independent with respect to this trivial pregeometry i.e. if for all $i$, $\tau_i$ is non-special and for all $i \neq j$ $\tau_i \notin G \tau_j$.  \newline

We then let $T$ be the following theory:
$$T_H \bigcup \Th(\langle \C , + , \cdot , 0 , 1 \rangle) \bigcup_{( g_1,...,g_n) \subset G} \Psi_{( g_1,...,g_n )} \bigcup_{s \in S} \tp(s)$$
where here $\tp(s)$ is the complete type of the special point $s$ in the standard model $\C_j$.

\section{Quantifier elimination and completeness}

\begin{pro}\label{QE}
Let
$$\cM = \langle \cH, \cF , j \rangle \textrm{ and } \cM' = \langle \cH', \cF , j' \rangle$$
be $\omega$-saturated models of $T$ and
$$\sigma : H\cup F \ra H' \cup F$$
a partial isomorphism with finitely generated domain $D$. Then given any $z \in H \cup F$, $\sigma$ extends to the substructure
generated by $D \cup \{z \}$.
\end{pro}

\begin{proof}
If $( h_1,...,h_n) =\bar{h} \in H^n$ generates $D\cap H$, then we may assume that $(\bar{h}, z)$ is $G$-independent. Suppose that $z \in H-D$ and let $C$ be a finite subset of $F$ such that $C \cup \{h_1,...,h_n \}$ generates $D$. Let 
$$L:=\left( \dcl^{\cM}(\emptyset)\cap  F \right) (C,j(G\bar{h}))\cong \left( \dcl^{\cM'}(\emptyset)\cap F' \right) (C^{\sigma},j(G \bar{h}^{\sigma})).$$
By \ref{fielddet}, $\qftp(z/ D))$ is determined by the union of all $\qftp_F((j(g_1 x), ..., j(g_n x))/L)$ over all finite tuples $(g_1,...,g_n) \subset G$, and by Corollary \ref{realise}, every finite subset of this type is realisable in any model of $T$. Therefore, by compactness the type is consistent, and since $\cM'$ is $\aleph_0$-saturated, it has a realisation $z' \in H'$. The case where $z \in F$ is covered by the above by \ref{fielddet} and since $j$ is surjective (we may assume that $z$ is non-special).
\end{proof}

\begin{cor}
$T$ is complete, has quantifier elimination and is superstable.
\end{cor}
 
\begin{rem}
By quantifier elimination we now know that $\dcl(\emptyset)=S \cup \Q(j(S))$ in any model of $\Th(j)$.
\end{rem}

\subsection{The $\omega$-saturated model $\M$}

We may view $\tc$ as a model of $\Th(j)$ (which we denote by $\M$) as follows:

$$\langle \langle H' ; G \rangle \ra^{\hat{j}} \langle \C ;+ \cdot,0,1 \rangle \rangle$$

As a set, $H'$ is defined to be $S\cup T$ where $T$ are the non-special fibres in $\tc$ (which has a natural action of $G$) and $\hat{j}$ is the projection down to $\psl2(\Z) \backslash \H$ composed with the standard $j$-function.

\begin{pro}
$\M$ is an $\omega$-saturated model of $\Th(j)$.
\end{pro}
\begin{proof}
If $x$ is special then its type is fixed so we only need worry about non-special elements. Consider the type of $x$ non-special, realised in the top sort of some model. We have seen that $\tp(x)$ is determined by $\Loc(j(G x))$ and this is determined by finite chunks of it. So  consider a finite tuple $(g_1,...,g_n) \subset G$. Then $\Loc(j(g_1 x),...,j(g_n x))$ is a subvariety of $V_{(g_1,...,g_n)}$ and so since $\M$ satifies the axiom $\Psi_{( g_1,...,g_n )}$, $\hat{j}$ maps onto $V_{(g_1,...,g_n)}$ and therefore onto $\Loc(j(g_1 x),...,j(g_n x))$.



\end{proof}

\begin{rem}
Now by basic model theory, given two countable models $\cM$ and $\cM'$ of $\Th(j) + SF$, we can embed them both in $\M$ (since $\M$ is $\omega$-saturated and you can embed a model of cardinality $\omega$ into an $\omega$-saturated model). In particular, given a tuples $\bar{x} \in H$ and $\bar{x}' \in H'$ we can embed $\cl_{\cM}(\bar{x})$ and $\cl_{\cM'}(\bar{x}')$ into $\M$ (where the colsure operator here is the one from section \ref{secsuf}). 
\end{rem}

\section{Necessary conditions}

In this section we show that if $\Th(j)+SF$ is $\aleph_1$-categorical, then the adelic Mumford-Tate conjecture for elliptic curves must hold. The argument goes as follows:
First we show that types of independent tuples (see definition below) in $\omega$-saturated models of $\Th(j)$, and in models of $\Th(j)+SF$ are the same, and that there are either finitely or uncountably many such types over a point in $\C^n$. Finally we use a result of Keisler saying that if we have $\aleph_1$-categoricity, then there can be only countably many such types, and then we translate this model theoretic statement into arithmetic geometry.

\begin{dfn}
By a \textit{strongly $G$-independent} tuple $( \tau_1,..., \tau_r )  \subset H$ we mean that $\tau_i$ is non-special and $j(\tau_i)$ is not in the Hecke orbit of $j(\tau_j)$ for all $i$ and $j$. These conditions can be summarised by stating that $\Phi_N(j(\tau_i),j(\tau_j)) \neq 0$ for all $N$. Note that in models of $\Th(j) + SF$, being independent is the same as being strongly independent.
\end{dfn}

\begin{pro}
Given an $\omega$-saturated model $\langle \cH,\cF \rangle$ of $\Th(j)$ and a strongly $G$-independent tuple $\bar{\tau} \subset H$, there is a model of $\Th(j)+SF$ realising $\tp(\bar{\tau})$. 
\end{pro}
\begin{proof}
We construct $\cN \models \Th(j) + SF$ using our axiomatization of $\Th(j)$ as follows: Take $j(G \bar{\tau})$, now take the field theoretic algebaic closure of this and look at the inverse image under $j$. Outside of $G \bar{\tau}$, this is a set of disjoint $G . \psl2(\hat{\Z})=\GL_2(\A_f)^{ad}$-orbits so choose one element in each orbit, close off under $G$ and take the image under $j$. Clearly this satisfies $SF$ so it remains to show that this is a model of $\Th(j)$. However, we know from the proof of quantifier elimination and completeness of the theory $T$ that we just really need to check that our new model $\cN$ satisfies $\Psi_{\langle g_1,...,g_n \rangle}$ for all $\langle g_1,...,g_n \rangle \subset G$ (and it clearly does). The type of the special points is fixed in every model since they are in $\dcl(\emptyset)$ (and non-standard special points don't appear since each special matrix fixes a unique $s \in \H$ and the theory knows this).
\end{proof}
Note that the stronger statement that if $\bar{\tau}$ is a $G$-independent tuple then we can realize $\tp(\bar{\tau})$ in a model of $\Th(j) + SF$ isn't true. To see this consider the type of $\bar{\tau} = \langle \tau_1, \tau_2 \rangle$ such that $j(\tau_1)=j(\tau_2)$ but $\tau_1 \notin \sl2(\Z) \tau_2$. \newline

So since a type in the pro-\'etale cover $\M$ is realised in a model of $\Th(j) + SF$, if the standard model $\C_j$ is the unique model of $\Th(j) + SF$ of cardinality continuum, then all 1-types realised in $\M$ are actually realised in $\C_j$. So we may interpret the categoricity statement as saying that when considered as structures in this language, the analytic universal cover actually contains all the information about 1-types contained in the pro-\'etale cover. However as we saw above, the property of a model having non-standard  fibres may be expressed by realising a certain 2-type.

\begin{pro}\label{number types}
Let $\bar{\tau} \in \H$ be independent. Then the number of types tuples $\bar{\tau}'$, realisable in models of $\Th(j) + SF$ such that $j'(\bar{\tau}')=j(\bar{\tau})$, is either finite or $2^{\aleph_0}$.
\end{pro}

\begin{proof} 
For a tuple $\bar{z} \subset \C$ and a subfield $K$ we let $\Loc(\bar{z} / K)$ be the minimal algebraic variety over $K$ containing $\bar{z}$ i.e. the `Weil locus'.  Consider $\bar{\tau}, \bar{\tau}' $ in an $\omega$-saturated model such that $j(\bar{\tau}) = j(\bar{\tau}')$, and suppose that 
$$\Loc(j(G_N \bar{\tau}) / \Q(j(\bar{\tau}))) \neq \Loc(j(G_N \bar{\tau}' ) / \Q(j(\bar{\tau}))).$$ 
As notation, now write $\Loc(X)$ for $\Loc(X / \Q(j(\bar{\tau})))$. We imagine a tree where branches are types of tuples $\bar{\tau}'$ such that $j(\bar{\tau}')=j(\bar{\tau})$.
 Now also suppose there is $\bar{\tau}''$ such that $$\Loc(j(G_N \bar{\tau}'')) = \Loc(j(G_N \bar{\tau})) 
\textrm{ but } \Loc(j(G_Nm \bar{\tau}))\neq \Loc(j(G_Nm \bar{\tau}'')).$$ We claim that there exists $\bar{\tau}'''$ such that $$\Loc(j(G_N \bar{\tau}'))=\Loc(j(G_N \bar{\tau}''')) \textrm{ and } \Loc(j(G_{Nm}\bar{\tau}'))\neq \Loc(j(G_{Nm} \bar{\tau}''')):$$ By \ref{definable cyclic} and \ref{definable etale}, there is a $\Q(j(\tau))$-definable function sending $\Loc(j(G_N \tau))$ to \linebreak $\Loc(j(G_N \tau'))$, and $\Loc(j(G_Nm \tau''))$ follows $\Loc(j(G_Nm \tau))$ across onto the other branch of the tree, proving the claim. Now we're done since our tree of types branches homogeneously.
\end{proof}

\begin{thm}[Keisler]
If an $\cL_{\omega_1 , \omega}$-sentence has less than the maximum number of models of cardinality $\aleph_1$ (e.g. is $\aleph_1$-categorical) then there are only countably many $\cL_{\omega_1 , \omega}$-types realisable over $\emptyset$.
\end{thm}
So it is really a stability condition coming from a restriction on the number of models, and not something specific to categoricity that we are going to translate into arithmetic geometry (in particular an instance of the Mumford-Tate conjecture), however categoricity becomes more significant when we show that Mumford-Tate implies categoricity in the next section.\newline

Now by Keisler's theorem, if we want $Th(j)+ SF$ to be $\aleph_1$-categorical, then there must be finitely many types as in \ref{number types} above. Let $\tilde{p}:\tc \ra \C$ be the projection. Then, as we noted above, the fibre $\tilde{p}^{-1}(j(\tau)) \subset \tc$ above the point $j(\tau)$ is a $\pi_1'$-torsor and we have a homomorphism
$$\rho: \Gal(\C / \Q(j(\tau))) \lora \pi_1'.$$
Now if Galois doesn't conjugate two elements in the fibre, then they correspond to two distinct types, so if we want there to be finitely many types then the image of $\rho$ must be of finite index in $\pi_1'$. Looking at the fibre of a tuple, we get:

\begin{thm}
If $\Th(j)+SF$ is $\aleph_1$-categorical then for all non-CM $\tau_1,...,\tau_n$, the image of the homomorphism
$$\rho: \Gal(\C / \Q(j(\tau),...,j(\tau_n))) \lora \pi_1'^n$$
is of finite index.
\end{thm}

\section{Sufficient conditions}\label{secsuf}





\begin{lem}[$\omega$-homogeneity over countable models and $\dcl(\emptyset)$]\label{homo}
Let
$$\cM = \langle F,H,j  \rangle \textrm{ and } \cM' = \langle F,H',j'  \rangle$$
be models of $T+SF$. Let $k \subset \C$ be a finitely generated extension of a countable algebraically closed field or a finitely generated extension of $\Q(j(S))$, and let $\bar{h}=\{h_1,...,h_n \}\subset H$ and $\bar{h'} \subset H'$ be such that $\qftp(\bar{h}/k)=\qftp(\bar{h'}/k)$. Then for all non-special $\tau \in H$ there is $\tau' \in H'$ such that $\qftp(\bar{h}\tau /k)=\qftp(\bar{h'}\tau' /k)$.
\end{lem}
\begin{proof}
The idea is that by Adelic Mumford-Tate (\ref{amt}), all of the information in $\qftp(\bar{h} \tau /k)$ is contained in the field type of a finite subset of the Hecke orbit, and we can find $\tau' \in H'$ which contains this `finite' amount of information by \ref{lem_locus_irreducible} (so the theory is atomic up to the failure of atomicity for $ACF_0$). Note that in general $\tp(ab)$ contains the same information as $\tp(a)\cup \tp(b / a)$. \newline

We may assume that the $h_i$ and $\tau$ are $G$-independent. Let $$K=k(j(h_1),...,j(h_n),j(\tau))$$ 
and
$$L:=K(\tor(A_{j(\bar{h})})) \cong K(\tor(A_{j(\bar{h}')}).$$
By \ref{amtapp} there exists an $m$ such that every automorphism $\sigma \in \Aut(T(E)) \cap \sl2(\hat{\Z})$ which fixes $E_{j(\tau)}[m]$ belongs to $\Gal(\overline{L} / L)$. By Lemma \ref{lem_locus_irreducible} we can find $\tau' \in H'$ such that $j(g \tau) = j(g \tau')$ for all $g \in G_{i}$ ($0 \leq i \leq m$) and we are done.

\end{proof}

\begin{proof}[Proof of \ref{thmcat}]
We appeal to the theory of quasiminimal excellence as in \cite[$\S1$]{kirby2010quasiminimal}. We define a closure operator $\cl:=j^{-1}\circ \acl \circ j$ which (by properties of $\acl$) is clearly a pregeometry with the countable closure property such that closed sets are models of $T+ SF$. Lemma \ref{homo} gives us Condition II.2 (of \cite[$\S1$]{kirby2010quasiminimal}), and all other conditions of quasiminimality (i.e. $0$, I.1, I.2, I.3 and II.1) are then easily seen to be satisfied. $\aleph_1$-categoricity of $T + SF$ follows (by \cite[Corollary 2.2]{kirby2010quasiminimal} for example). \newline

By \ref{homo}, the standard model $\C_j$ with the pregeometry $\cl$ is a quasiminimal pregeometry structure (as in \cite[\S2]{bays2012quasiminimal}) and  $\cK(\C_{j})$ is a quasiminimal class, so by the main result of \cite{bays2012quasiminimal}  (Theorem 2.2), $\cK(\C_j)$ has a unique model (up to isomorphism) in each infinite cardinality, and in particular  $\cK(\C_{j})$ contains a unique structure $\cC_{j}$ of cardinality $\aleph_0$. Let $\cK$ be the class of models of $T+ SF + \trdeg \geq \aleph_0$. It is clear that $\cK(\C_j) \subseteq \cK$ since $\C_j \in \cK$, and to prove the theorem we want to show that $\cK = \cK(\C_{j})$. \newline

By \ref{homo}, $\cK$ has a unique model $\cM$ of cardinality $\aleph_0$. Since $\cK$ is the class of models of an $\cL_{\omega_1,\omega}$-sentence, $\cK$ together with closed embeddings is an abstract elementary class with Lowenheim-Skolem number $\aleph_0$, so by downward Lowenheim-Skolem (in $\cK$), everything in $\cK$ is a direct limit (with elementary embeddings as morphisms) of copies of the unique model of cardinality $\aleph_0$. Finally, all the embeddings in $\cK$ are closed with respect to the pregeometry, so $\cK =\cK(\cM)=\cK(\cC_j)= \cK(\C_{j})$.
\end{proof}

\section{A final remark}
It should be noted that while the structure $\C_j$ is a natural one to study, it turns out that this is not the most natural language for looking at the $j$-function. The terminology and results of this note are ugly. For example, categoricity being equivalent to a slight weakening of the adelic Mumford-Tate conjecture, along with the fact that Galois action on arbitraty products of elliptic curves (i.e. including those with complex multiplication) does not enter into the picture is frustrating. \newline

The essence of what we are doing here, is looking at the universal cover of the curve $\P^1(\C)-\{0,1,\infty \}$, and it would be much more pleasing at this stage if representations went into the full \'etale fundamental group, and that categoricity was equivalent to Galois representations being of open image in the Hodge groups of arbitrary products of elliptic curves. Note that main content of the adelic Mumford-Tate conjecture with respect to the abelian variety is the image being open in the Hodge group, the rest comes from the cyclotomic character, which is really an issue with $\G_m$, so categoricity being equivalent to the representations being open in the Hodge group and not the Mumford-Tate group would be expected. These kinds of issues will be addressed in my upcoming PhD thesis. That being said, I hope that this presentation of the work is worth while, since it paves the way for the construction of a `pseudo $j$-function' in analogy with Zilber's pseudoexponential field.

\newpage


\end{document}